\documentclass[12pt,reqno]{amsart}
\usepackage{enumerate, latexsym, amsmath, amsfonts, amssymb, amsthm, graphicx, color}
\def\pmod #1{\ ({\rm{mod}}\ #1)}
\def\Z{\Bbb Z}
\def\N{\Bbb N}

\def\Q{\Bbb Q}

\def\R{\Bbb R}

\def\l{\left}
\def\r{\right}
\def\bg{\bigg}
\def\({\bg(}
\def\){\bg)}
\def\t{\text}
\def\f{\frac}

\def\ord{{\rm ord}}

\def\gen{{\rm gen}}
\def\Aut{{\rm Aut}}

\def\ls{\leqslant}
\def\gs{\geqslant}

\def\sm{\setminus}

\def\ve{\varepsilon}

\def\eq{\equiv}

\def\da{\delta}

\theoremstyle{plain}
\newtheorem{theorem}{Theorem}

\newtheorem{lemma}{Lemma}

\theoremstyle{definition}

\theoremstyle{remark}
\newtheorem{remark}{Remark}

\makeatletter
\@namedef{subjclassname@2010}{%
  \textup{2010} Mathematics Subject Classification}
\makeatother
 \vspace{4mm}

\begin{document}
 \baselineskip=17pt
\hbox{Nanjing Univ. J. Math. Biquarterly 40 (2023), no.\,1, 54--71.}
\medskip
\title[]
{Arithmetic progressions represented by diagonal ternary quadratic forms}
\date{}
\author[Hai-Liang Wu and Zhi-Wei Sun] {Hai-Liang Wu and Zhi-Wei Sun}

\thanks{
\newline \indent The research was supported by the National Natural Science
Foundation of China (Grant No. 11971222), and the initial version was posted to arXiv in 2018 with the ID
				{\tt arXiv:1811.05855}.}

\address {(Hai-Liang Wu)   School of Science, Nanjing University of Posts and
Telecommunications, Nanjing 210023, People's Republic of China}
\email{\tt whl.math@smail.nju.edu.cn}

\address {(Zhi-Wei Sun) Department of Mathematics, Nanjing
University, Nanjing 210093, People's Republic of China}
\email{{\tt zwsun@nju.edu.cn}}

\begin{abstract} Let $d>r\gs0$ be integers. For positive integers $a,b,c$, if any term of the arithmetic progression
$\{r+dn:\ n=0,1,2,\ldots\}$ can be written as $ax^2+by^2+cz^2$ with $x,y,z\in\Z$, then the form $ax^2+by^2+cz^2$ is called $(d,r)$-universal.
In this paper, via the theory of ternary quadratic forms we study the $(d,r)$-universality of some diagonal ternary quadratic forms
conjectured by L. Pehlivan and K. S. Williams, and Z.-W. Sun.
For example, we prove that $2x^2+3y^2+10z^2$ is $(8,5)$-universal, $x^2+3y^2+8z^2$ and $x^2+2y^2+12z^2$ are $(10,1)$-universal and $(10,9)$-universal,
and $3x^2+5y^2+15z^2$ is $(15,8)$-universal.
\end{abstract}

\maketitle

\section{Introduction}
\setcounter{lemma}{0}
\setcounter{theorem}{0}
\setcounter{corollary}{0}
\setcounter{remark}{0}
\setcounter{equation}{0}
\setcounter{conjecture}{0}

Let $\N=\{0,1,2,\ldots\}$. The Gauss-Legendre theorem on sums of three squares states that
$\{x^2+y^2+z^2:\ x,y,z\in\Z\}=\N\sm\{4^k(8l+7):\ k,l\in\N\}$. A classical topic in the study of number theory asks, given a quadratic polynomial $f$ and an integer $n$, how can we decide when $f$ represents $n$ over the integers? This topic has been extensively investigated. It is known that
for any $a,b,c\in\Z^+=\{1,2,3,\ldots\}$ the exceptional set
$$E(a,b,c)=\N\sm\{ax^2+by^2+cz^2:\ x,y,z\in\Z\}$$
is infinite, see, e.g., \cite{DW}.

An integral quadratic form $f$ is called {\it regular} if it represents
each integer represented by the genus of $f$. L. E. Dickson \cite[pp. 112-113]{D39} listed all the $102$ regular ternary quadratic forms $ax^2+by^2+cz^2$ together with the explicit characterization of
$E(a,b,c)$, where
$1\ls a\ls b\ls c\in\Z^+$ and $\gcd(a,b,c)=1$. In this direction, W. C. Jagy, I. Kaplansky and A. Schiemann \cite{JKS} proved that there are at most $913$ regular positive definite integral ternary quadratic forms.

By the Gauss-Legendre theorem, for any $n\in\N$ we can write $4n+1=x^2+y^2+z^2$ with $x,y,z\in\Z$.
It is also known that for any $n\in\N$ we can write $2n+1$ as $x^2+y^2+2z^2$ (or $x^2+2y^2+3z^2$, or $x^2+2y^2+4z^2$) with $x,y,z\in\Z$ (see, e.g., Kaplansky \cite{odd}). Thus, it is natural to introduce the following definition.

{\it Definition} 1.1. Let $d\in\Z^+$ and $r\in\{0,\ldots,d-1\}$. For $a,b,c\in\Z$, if any $dn+r$ with $n\in\N$
can be written as $ax^2+by^2+cz^2$ with $x,y,z\in\Z$, then we say that the ternary quadratic form $ax^2+by^2+cz^2$
is $(d,r)$-{\it universal}.

In 2008, A. Alaca, S. Alaca and K. S. Williams \cite{AAW} proved that there is no binary  positive definite quadratic form  which can represent all nonnegative integers in a residue class.

Z.-W. Sun \cite{S15} proved that $x^2+3y^2+24z^2$ is $(6,1)$-universal. Moreover, in 2017 he \cite[Remark 3.1]{S17}
confirmed his conjecture that for any $n\in\Z^+$ and $\da\in\{0,1\}$ we can write
$6n+1$ as $x^2+3y^2+6z^2$ with $x,y,z\in\Z$ and $x\eq\da\pmod2$. This implies that
$4x^2+3y^2+6z^2$ and $x^2+12y^2+6z^2$ are $(6,1)$-universal.
On August 2, 2017, Sun \cite{S17o} published on OEIS his list (based on his computation)
of all possible candidates of $(d,r)$-universal
irregular ternary quadratic forms $ax^2+by^2+cz^2$ with $1\ls a\ls b\ls c$ and $3\ls d\ls 30$. For example,
he conjectured that
$$x^2+3y^2+7z^2,\ x^2+3y^2+42z^2,\ x^2+3y^2+54z^2$$
are all $(6,1)$-universal, $x^2+7y^2+14z^2$ is $(7,1)$-universal
and $x^2+2y^2+7z^2$ is $(7,r)$-universal for each $r=1,2,3$.
In 2018 L. Pehlivan and K. S. Williams \cite{PK} also investigated such problems independently, actually
they studied $(d,r)$-universal quadratic forms $ax^2+by^2+cz^2$ with $1\ls a\ls b\ls c$ and $3\ls d\ls 11$.

Pehlivan and Williams \cite{PK} considered the $(8,1)$-universality of  $x^2+8y^2+24z^2$, $x^2+2y^2+64z^2$ and $x^2+8y^2+64z^2$ open.
However, B. W. Jones and G. Pall \cite{Jones} proved in 1939 that for any $n\in\N$ we can write
$$8n+1=x^2+8y^2+64z^2=x^2+2(2y)^2+64z^2$$
with $x,y,z\in\Z$, and hence $x^2+2y^2+64z^2$ and $x^2+8y^2+64z^2$ are indeed $(8,1)$-universal.
As $8x(x+1)/2+1=(2x+1)^2$, the $(8,1)$-universality of $x^2+8y^2+24z^2$ yields $\{x(x+1)/2+y^2+3z^2:\ x,y,z\in\Z\}=\N,$
which was conjectured by Sun \cite{S07} and confirmed in \cite{GPS}.

The first part and part (ii) with $i\in\{2,3\}$ of the following result were conjectured by Pehlivan and Williams \cite{PK}, as well as Sun \cite{S17o}.

\begin{theorem}\label{Th B}
{\rm (i)} The form $2x^2+3y^2+10z^2$ is $(8,5)$-universal.

{\rm (ii)} Let $n\in\Z^+$, $\da\in\{1,9\}$ and $i\in\{1,2,3\}$. Then $10n+\delta=x_1^2+2x_2^2+3x_3^2$
for some $(x_1,x_2,x_3) \in \Z^3$ with $2\mid x_i$.
\end{theorem}

In the spirit of Sun \cite{List}, it is easy to see that Theorem 1.1(i) has the following equivalent form: For any $n\in\N$ there are $x,y,z\in\Z$ and $\da\in\{0,1\}$ such that $n=x(x+1)+3y(y+1)/2+5z(z+\delta)$.

Kaplansky \cite{odd} showed that there are at most 23 positive definite integral ternary quadratic forms that can represent all positive odd integers
 (19 for sure and 4 plausible candidates, see also Jagy \cite{Jagy96} for further progress).
  Using one of the 19 forms, we obtain the following result originally conjectured by Sun \cite{S17o}.

\begin{theorem}\label{Th C}
The forms $x^2+3y^2+14z^2$ and $2x^2+3y^2+7z^2$ are both $(14,7)$-universal.
\end{theorem}

Now we turn to study Sun's conjectural $(15,r)$-universality of some positive definite integral ternary quadratic forms.

\begin{theorem}\label{Th D}
{\rm (i)} For any $n\in\N$ and $i\in\{1,2,3\}$,
there exists $(x_1,x_2,x_3)\in\Z^3$ with $3\mid x_i$ such that $15n+5=2x_1^2+3x_2^2+5x_3^2$.
\medskip

{\rm (ii)} The form $x^2+y^2+15z^2$ is $(15,5r)$-universal for $r=1,2$, and $3x^2+3y^2+5z^2$ is $(15,5)$-universal.

{\rm (iii)} For any $r=1,2$, both $x^2+y^2+30z^2$ and $2x^2+3y^2+5z^2$ are $(15,5r)$-universal. Also,
the forms $x^2+6y^2+15z^2$ and $3x^2+3y^2+10z^2$ are $(15,10)$-universal.

{\rm (iv)} The form $x^2+2y^2+15z^2$ is $(15,3r)$-universal for each $r=1,2,3,4$, and the form $3x^2+5y^2+10z^2$ is $(15,3r)$-universal for $r=1,4$.
\end{theorem}

\begin{theorem}\label{Th E}
{\rm (i)} The form $x^2+3y^2+5z^2$ is $(15,3r)$-universal for each $r=1,2,3,4$.
Also, $x^2+5y^2+15z^2$ is $(15,3r)$-universal for $r=2,3$.
\medskip

{\rm (ii)} The form $x^2+3y^2+15z^2$ is $(15,r)$-universal for each $r\in \{1,7,13\}$.
Also, the form $x^2+15y^2+30z^2$ is $(15,r)$-universal for $r=1,4$, and
the form $x^2+10y^2+15z^2$ is $(15,r)$-universal for all $r\in \{4,11,14\}$.
\medskip

{\rm (iii)} The form $3x^2+5y^2+6z^2$ is $(15,r)$-universal for each $r\in \{8,11,14\}$. Also, $3x^2+5y^2+15z^2$ and $3x^2+5y^2+30z^2$ are both $(15,8)$-universal.
\end{theorem}
\begin{remark}
Our proof of Theorem \ref{Th E} relies heavily on the genus theory of quadratic forms as well as
the Siegel-Minkowski formula.
\end{remark}

We will give a brief overview of the theory of ternary quadratic forms in the next section, and show
Theorem \ref{Th B}-\ref{Th E} in Sections 3-5 respectively.

\maketitle
\section{Some preparations}
\setcounter{lemma}{0}
\setcounter{theorem}{0}
\setcounter{corollary}{0}
\setcounter{remark}{0}
\setcounter{equation}{0}
\setcounter{conjecture}{0}

Let
\begin{equation}\label{2.1}
f(x,y,z)=ax^2+by^2+cz^2+ryz+szx+txy
\end{equation}
be a positive definite ternary quadratic form with integral coefficients. Its associated matrix is
$$A=\begin{pmatrix} 2a & t &s \\t & 2b &r \\ s & r &2c \end{pmatrix}.$$
The discriminant of $f$ is defined by $d(f):=\det(A)/2$.

The following lemma is a fundamental result on integral representations of quadratic forms (cf. \cite [pp.129]{C}).

\begin{lemma}\label{Lem2.1}
Let $f$ be a nonsingular integral quadratic form and let $m$ be a nonzero integer represented by $f$ over the real field $\R$
and the ring $\Z_p$ of $p$-adic integers for each prime $p$.
Then $m$ is represented by some form $f^*$ over $\Z$ with $f^*$ in the same genus of $f$.
\end{lemma}

Now, we introduce some standard notations in the theory of quadratic forms which can be found in \cite{C, Ki, Oto}.
For the positive definite ternary quadratic form $f$ given by (\ref{2.1}) and $n\in\N$, we write
 $$r(n,f):=|\{(x,y,z)\in \Z^3 : f(x,y,z)=n\}|$$
 (where $|S|$ denotes the cardinality of a set $S$), and let
  \begin{equation*}r(n,\gen(f)):=\(\sum_{f^*\in \gen(f)}\frac{1}{|\Aut(f^*)|}\)^{-1}\sum_{f^*\in
  \gen(f)}\frac{r(n,f^*)}{|\Aut(f^*)|},
  \end{equation*}
  where the summation is over a set of representatives of the classes in $\gen(f)$,
  and $\Aut(f^*)$ is the group of integral isometries of $f^*$.

We also need the following result obtained from the Siegel-Minkowski formula
and the knowledge of local densities.

\begin{lemma}{\rm (\cite[Lemma 4.1]{Wu-Sun})}\label{Lem2.3}
Let $f$ be a positive ternary quadratic form with discriminant $d(f)$. Let $m\in\{1,2\}$ and suppose that $m$ is represented by $\gen(f)$. Then, for each prime $p\nmid 2md(f)$, we have
\begin{equation}
\frac{r(mp^2,\gen(f))}{r(m,\gen(f))}=p+1-\left(\frac{-md(f)}{p}\right),
\end{equation}
where $(\f{\cdot}p)$ is the Legendre symbol.
\end{lemma}

\maketitle
\section{Proof of Theorem 1.1}
\setcounter{lemma}{0}
\setcounter{theorem}{0}
\setcounter{corollary}{0}
\setcounter{remark}{0}
\setcounter{equation}{0}
\setcounter{conjecture}{0}

\begin{lemma}\label{Lem4.2}
For any $n\in\Z^+$ and $\delta\in\{1,9\}$, we can write $10n+\delta=x^2+2y^2+3z^2$ with $x,y,z\in\Z$ and $y^2+z^2\ne0$.
\end{lemma}
\begin{proof}
By \cite [pp.112--113]{D39} we can write $10n+\delta=x^2+2y^2+3z^2$ with $x,y,z\in\Z$; if $10n+\delta$ is not a square then $y^2+z^2$ is obviously nonzero.

Now suppose that $10n+\delta=m^2$ for some $m\in\N$. As $n>0$, we have $m>1$.

{\it Case} 1. $m$ has a prime factor $p>3$.

In this case, by Lemma \ref{Lem2.3} we have
\begin{equation*}
r(p^2,x^2+2y^2+3z^2)=2\l(p+1-\left(\frac{-6}{p}\right)\r).
\end{equation*}
Hence, $r(m^2,x^2+2y^2+3z^2)\gs r(p^2,x^2+2y^2+3z^2)>2$. Thus, for some $(r,s,t)\in\Z^3$ with
$s^2+t^2\ne0$ we have $10n+\delta=m^2=r^2+2s^2+3t^2$.

{\it Case} 2. $10n+\delta=m^2=3^{2k}$ with $k\in\Z^+$.

In this case,
\begin{equation*}
10n+\da=3^{2k}=(2\times3^{k-1})^2+2\times(3^{k-1})^2+3\times(3^{k-1})^2.
\end{equation*}

In view of the above, we have completed the proof.
\end{proof}
\medskip

\begin{lemma}\label{Lem4.3}
If $n=2x^2+3y^2>0$ with $x,y\in\Z$ and $5\mid n$,
then we can write $n=2u^2+3v^2$ with $u,v\in\Z$ and $5\nmid uv$.
\end{lemma}
\begin{proof} We use induction on $k=\ord_5(\gcd(x,y))$, the $5$-adic order of the greatest common divisor of $x$ and $y$.

When $k=0$, the desired result holds trivially.

Now let $k\gs1$ and assume the desired result for smaller values of $k$.
Write $x=5^kx_0$ and $y=5^ky_0$, where $x_0$ and $y_0$ are integers not all divisible by $5$.
Then $x_0+6y_0$ or $x_0-6y_0$ is not divisible by $5$. Hence we may choose $\ve\in\{\pm1\}$ such that
$5\nmid x_0+6\ve y_0$. Set $x_1=5^{k-1}(x_0+6\ve y_0)$ and $y_1=5^{k-1}(4x_0-\ve y_0)$. Then $\ord_5(\gcd(x_1,y_1))=k-1$.
Note that \begin{align*}
5^{2k}(2x_0^2+3y_0^2)=5^{2k-2}(2(x_0+6\ve y_0)^2+3(4x_0-\ve y_0)^2)=2x_1^2+3y_1^2.
\end{align*}
So, applying the induction hypothesis we immediately obtain the desired result.
\end{proof}
\medskip

\noindent{\it Proof of Theorem} 1.1. (i) It is easy to see that $8n+5$ can be represented by the genus of
$f(x,y,z)=2x^2+3y^2+10z^2$. There are two classes in the genus of $f$, and the one not containing $f$ has the representative $g(x,y,z)=3x^2+5y^2+5z^2+2yz-2zx+2xy$.
It is easy to verify the following identity:
\begin{equation} \label{4.3}
f\l(\frac{x}{2}+y-z,\ y+z,\ \frac{x}{2}\r)=g(x,y,z).
\end{equation}

Suppose that $8n+5=g(x,y,z)$ for some $x,y,z\in\Z$. Then
$$1\equiv 8n+5=g(x,y,z)\equiv 3x^2+(y+z)^2+2x(y-z)\pmod 4.$$
 Hence $y\not\equiv z\pmod 2$ and $2\mid x$.
 In light of the identity (\ref{4.3}), $8n+5$ is represented by $f$ over $\Z$.

 By Lemma \ref{Lem2.1} and the above,  $8n+5$ can be represented by $2x^2+3y^2+10z^2$ over $\Z$.

(ii) Let $h(x,y,z)=x^2+2y^2+3z^2$.
By \cite [pp.112--113]{D39}, we can write $10n+\delta=h(x,y,z)$ for some $x,y,z\in\Z$.

We claim that there are $u,v,w\in\Z$ with $u-2v+4w\equiv 0\pmod 5$ such that $10n+\delta=h(u,v,w)$. Here we handle the case $\delta=1$. (The case $\da=9$ can be handled similarly.)

{\it Case} 1. $x^2\equiv -1\pmod 5$.

 It is easy to see that $y^2\equiv 0\pmod 5$ and $z^2\equiv -1\pmod 5$, or $y^2\equiv 1\pmod 5$ and $z^2\equiv 0\pmod 5$.
When $y^2\equiv 0\pmod 5$ and $z^2\equiv -1\pmod 5$, without loss of generality we may assume that $z\equiv x\pmod 5$ (otherwise, we may replace $z$ by $-z$). If $y^2\equiv 1\pmod 5$ and
$z^2\equiv 0\pmod 5$, then we simply assume  $y\equiv -2x\pmod 5$ without loss of generality. Note that our choice of $y$ and $z$ meets the requirement $x-2y+4z\eq0\pmod5$.

{\it Case} 2. $x^2\equiv 0\pmod 5$.

Clearly, we have $y^2\equiv -1\pmod 5$ and $z^2\equiv 1\pmod 5$. Without loss of generality, we may assume that $y\equiv 2z\pmod 5$ and hence $x-2y+4z\eq0\pmod5$.

{\it Case} 3. $x^2\equiv 1\pmod 5$.

Apparently, we have $y^2\equiv z^2\pmod 5$. By Lemmas \ref{Lem4.2} and \ref{Lem4.3}, we may simply assume that $5\nmid yz$.
When $y^2\equiv z^2\equiv x^2\equiv 1\pmod 5$, without loss of generality we may assume that $x\equiv y\equiv -z\pmod 5$.
If $y^2\equiv z^2\equiv (2x)^2\equiv -1\pmod 5$, then  we may assume that
$y\equiv z\equiv 2x\pmod 5$ without any loss of generality. So, in this case our choice of $y$ and $z$ also meets the requirement $x-2y+4z\eq0\pmod5$.

In view of the above analysis, we may simply assume $x-2y+4z\eq0\pmod5$ without any loss of generality. Note that
$h(x,y,z)=h(x^*,y^*,z^*)$, where
\begin{align*}z^*=&\frac{x-2y+4z}5\not\eq z\pmod2,
\\x^*=&2y-z+2z^*\not\eq x\pmod2,
\\y^*=&y-3z+3z^*\not\eq y\pmod2.
\end{align*}
So we have the desired result in part (ii) of Theorem \ref{Th B}.
\qed

\section{Proofs of Theorems 1.2-1.3}
\setcounter{lemma}{0}
\setcounter{theorem}{0}
\setcounter{corollary}{0}
\setcounter{remark}{0}
\setcounter{equation}{0}
\setcounter{conjecture}{0}

\noindent{\it Proof of Theorem} 1.2. By  \cite{odd}, we can write $2n+1=F(r,s,t)$ with
$r,s,t\in\Z$, where $F(x,y,z)=x^2+3y^2+2yz+5z^2$.
Since
$$
(2r-3t)^2+3(r+2t)^2+14s^2=7F(r,s,t)$$ and
$$2(s+3t)^2+3(2s-t)^2+7r^2=7F(r,s,t),$$
we see that $7(2n+1)$ is represented by the form $x^2+3y^2+14z^2$ as well as the form $2x^2+3y^2+7z^2$. \qed
\medskip

\noindent{\it Proof of Theorem} 1.3. (i) By \cite [pp.112--113]{D39}, we may write $3n+1=r^2+s^2+6t^2$ with
$r,s,t\in\Z$. One may easily verify the following identities:
\begin{align*}
5(r^2+s^2+6t^2)&=2(r\pm 3t)^2+3(r\mp 2t)^2+5s^2
\\&=2(s\pm 3t)^2+3(s\mp 2t)^2+5r^2.
\end{align*}
As exactly one of $r$ and $s$ is divisible by $3$, one of the the four numbers $r\pm2t$ and $s\pm2t$ is a multiple of $3$. This proves part (i) of Theorem 1.3.

(ii) Let $r\in\{1,2\}$. By \cite [pp.112--113]{D39}, for some $x,y,z\in\Z$ we have
$3n+r=x^2+y^2+3z^2$. Hence
\begin{equation*}
15n+5r=5(x^2+y^2+3z^2)=(x+2y)^2+(2x-y)^2+15z^2.
\end{equation*}

By \cite [pp.112--113]{D39}, we may write $3n+1=u^2+3v^2+3w^2$ with $u,v,w\in\Z$. Thus
\begin{equation*}
15n+5=5(u^2+3v^2+3w^2)=3(v+2w)^2+3(2v-w)^2+5u^2.
\end{equation*}

(iii) Let $r\in\{1,2\}$. By \cite [pp.112--113]{D39}, there are $u,v,w\in\Z$ such that
$3n+r=r^2+s^2+6t^2$. There are two classes in the genus of the form $x^2+y^2+30z^2$, and
the one not containing $x^2+y^2+30z^2$ has a representative $2x^2+3y^2+5z^2$. Since
\begin{align*}
15n+5r&=5(x^2+y^2+6z^2)
\\&=(x+2y)^2+(2x-y)^2+30z^2
\\&=2(x+3z)^2+3(x-2z)^2+5y^2,
\end{align*}
we see that $15n+5r$ is represented by $x^2+y^2+30z^2$ as well as $2x^2+3y^2+5z^2$.

By \cite [pp.112--113]{D39}, we can write $3n+2=2u^2+3v^2+3w^2$ with $u,v,w\in\Z$. There are two classes in the genus of
$x^2+6y^2+15z^2$, and the one not containing $x^2+6y^2+15z^2$ has a representative $3x^2+3y^2+10z^2$. As
\begin{align*}
15n+10&=5(2u^2+3v^2+3w^2)
\\&=(2u+3v)^2+6(u-v)^2+15w^2
\\&=3(u+2v)^2+3(2v-w)^2+10u^2,
\end{align*}
we see that $15n+10$ is represented by $x^2+6y^2+15z^2$ as well as $3x^2+3y^2+10z^2$.

(iv) Let $r\in\{1,2,3,4\}$. By \cite [pp.112--113]{D39}, there are $x,y,z\in\Z$ such that
$5n+r=x^2+2y^2+5z^2$. Hence
\begin{equation*}
15n+3r=3(x^2+2y^2+5z^2)=(x-2y)^2+2(x+y)^2+15z^2.
\end{equation*}

Now let $r\in\{1,4\}$.  By \cite [pp.112--113]{D39},
there are $u,v,w\in\Z$ such that $5n+r=u^2+5v^2+10w^2$. Thus, \begin{equation*}
15n+3r=3(u^2+5v^2+10w^2)=3u^2+5(v-2w)^2+10(v+w)^2.
\end{equation*}

In view of the above, we have completed the proof of Theorem 1.3. \qed

\section{Proof of Theorem 1.4}
\setcounter{lemma}{0}
\setcounter{theorem}{0}
\setcounter{corollary}{0}
\setcounter{remark}{0}
\setcounter{equation}{0}
\setcounter{conjecture}{0}

\noindent{\it Proof of Theorem} 1.4(i). Let $r\in\{1,2,3,4\}$. It is easy to see that $15n+3r$ can be represented by
$f_1(x,y,z)=x^2+3y^2+5z^2$ locally. There are two classes in the genus of $f_1$,
and the one not containing $f_1$ has a representative $f_2(x,y,z)=x^2+2y^2+8z^2-2yz$. One may easily verify
the following identities:
\begin{align}
f_1\l(\frac{x-y-z}{3}-2z,\ \frac{x-y-z}{3}+y,\ \frac{x-y-z}{3}+z\r)=&f_2(x,y,z)\label{5.1},
\\f_1\l(\frac{x+y+z}{3}+2z,\ \frac{x+y+z}{3}-y,\ \frac{x+y+z}{3}-z\r)=&f_2(x,y,z)\label{5.2}.
\end{align}

Suppose that $15n+3r=f_2(x,y,z)$ with $x,y,z\in\Z$. Then
$$x^2-(y+z)^2\eq f_2(x,y,z)\eq0\pmod3,$$
and hence $(x-y-z)/3$ or $(x+y+z)/3$ is an integer.
Therefore, by (\ref{5.1}), (\ref{5.2}) and Lemma \ref{Lem2.1}, we obtain that $x^2+3y^2+5z^2$ is
$(15,3r)$-universal.

Now let $r\in\{2,3\}$. One can easily verify that $15n+3r$ is represented by the genus of
$g_1(x,y,z)=x^2+5y^2+15z^2$. There are two classes in the genus of $g_1$, and the one not containing $g_1$ has
a representative $g_2(x,y,z)=4x^2+4y^2+5z^2+2xy$.
It is easy to verify the identity
\begin{equation}\label{5.3}
g_1\l(y+\frac{x+y\mp5z}{3},\ x-\frac{x+y\pm z}{3},\ \frac{x+y\pm z}{3}\r)=g_2(x,y,z).
\end{equation}
If $15n+3r=g_2(x,y,z)$ with $x,y,z\in\Z$, then
$$(x+y)^2-z^2\eq g_2(x,y,z)\eq0\pmod 3$$
Thus, with the help of (\ref{5.3}) and Lemma {\ref{Lem2.1}}, we obtain the desired result. \qed
\medskip

\begin{lemma}\label{Lem5.1} {\rm (Oh \cite{Oh})}
Let $V$ be a positive definite ternary quadratic space over $\Q$. For any isometry $T\in O(V)$ of infinite order,
$$V_T=\{x\in V: \text {there is a positive integer k such that}\  T^k(x)=x\}$$
is a subspace of $V$ of dimension one, and $T(x)=\det(T)x$ for any $x\in V_T$.
\end{lemma}
\begin{remark}\label{Rem5.1}
Unexplained notations of quadratic space can be found in \cite{C,Ki,Oto}.
\end{remark}

\begin{lemma}\label{Lem5.2}
Let $n\in\N$ and $r\in\{1,7,13\}$. If we can write $15n+r=f_2(x,y,z)=3x^2+4y^2+4z^2+2yz$ with $x,y,z\in\Z$,
then there are $u,v,w\in\Z$ with $u+2v-2w\not\equiv 0\pmod 3$ such that $15n+r=f_2(u,v,w)$.
\end{lemma}
\begin{proof}
Suppose that every integral solution of the equation $f_2(x,y,z)=15n+r$ satisfies $x+2y-2z\equiv 0\pmod 3$.
We want to deduce a contradiction.

Let
$$T=\begin{pmatrix} 1/3 & 2/3 &-2/3 \\-2/3 & 2/3 &1/3 \\ 2/3 & 1/3 &2/3 \end{pmatrix},$$
and let $V$ be the quadratic space corresponding to $f_2$.
Since
\begin{equation}\label{5.4}
f_2\l(\frac{x+2y-2z}{3}, -x+z+\frac{x+2y-2z}{3}, x+y-\frac{x+2y-2z}{3}\r)=f_2(x,y,z),
\end{equation}
we have $T\in O(V)$.
One may
easily verify that the order of $T$ is infinite and the space $V_T$ defined in Lemma \ref{Lem5.1} coincides with $\{(0,t,t): t\in\Q\}$.
As $15n+r\ne f_2(0,t,t)$ for any $t\in\Z$, we have $15n+r=f_2(x_0,y_0,z_0)$ for some
$(x_0,y_0,z_0)\in\Z^3\sm V_T$. Clearly, the set $\{T^k(x_0,y_0,z_0): k\gs0\}$ is infinite and its elements are solutions to the equation $f_2(x,y,z)=15n+r$.
This leads a contradiction since the number of integral representations of any integer by a positive
quadratic forms is finite.
\end{proof}

\begin{lemma}\label{Lem5.3} {\rm (Jagy \cite{Jagy14})} If $n=2x^2+2xy+3y^2$ $(x,y\in\Z)$ is a positive integer divisible by $3$, then there are $u,v\in\Z$
with $3\nmid uv$ such that $n=2u^2+2uv+3v^2$.
\end{lemma}
\medskip

The following lemma is a known result, see, e.g.,  \cite{Jagy14,Jones,Oh,S17}.

\begin{lemma}\label{Lem5.4}
If $n=x^2+y^2$ $(x,y\in\Z)$ is a positive integer divisible by $5$,
then $n=u^2+v^2$ for some $u,v\in\Z$ with $5\nmid uv$.
\end{lemma}
\medskip

\noindent{\it Proof of Theorem} 1.4(ii). (a) For each $r\in\{1,7,13\}$, it is easy to see that $15n+r$ can be represented by the genus of $f_1(x,y,z)=x^2+3y^2+15z^2$. There are two classes in the genus of $f_1(x,y,z)$, and the one not containing $f_1$ has a representative $f_2(x,y,z)=3x^2+4y^2+4z^2+2yz$. One may easily verify the following identities:
\begin{align}
f_1\l(x-y+z,\ \frac{x-2y}{3}-z,\ \frac{x+y}{3}\r)=&f_2(x,y,z)\label{5.5},
\\f_1\l(x+y-z,\ \frac{x-2z}{3}-y,\ \frac{x+z}{3}\r)=&f_2(x,y,z)\label{5.6}.
\end{align}
Suppose that $15n+r=f_2(x,y,z)$ for some $x,y,z\in\Z$.
As $3\nmid y$ or $3\nmid z$, when $3\nmid x$ we may assume that $(x+y)(x+z)\eq0\pmod3$ (otherwise we may replace $x$ by $-x$)
without loss of generality. If $3\mid x$ and $y\not\equiv z\pmod 3$, then $3\mid yz$ and hence $(x+y)(x+z)\eq0\pmod3$.
In the remaining case $3\mid x$ and $y\equiv z\pmod 3$, we have $x+2y-2z\eq0\pmod3$; however, we may apply Lemma \ref{Lem5.2}
to choose integers $u,v,w\in\Z$ so that $15n+r=f_2(u,v,w)$ and $u+2v-2w\not\eq0\pmod3$.

In view of the above analysis, there always exist $u,v,w\in\Z$ with $(u+v)(u+w)\equiv0\pmod3$
such that
$15n+r=f_2(u,v,w)$. With the help of (\ref{5.5}), (\ref{5.6}), and Lemma \ref{Lem2.1}, we obtain the $(15,r)$-universality of
$x^2+3y^2+15z^2$.

(b) Let $r\in\{1,4\}$. One can easily verify that $15n+r$ can be represented by
$g_1(x,y,z)=x^2+15y^2+30z^2$ locally. There are two classes in the genus of $g_1$, and the one not containing
$g_1$ has a representative $g_2(x,y,z)=6x^2+9y^2+10z^2-6xy$.

Suppose that $15n+r=g_2(x,y,z)$ with $x,y,z\in\Z$.
Clearly $3\nmid z$.
Since $15n+r\ne10z^2$, by Lemma \ref{Lem5.3} we may  assume that $x$ and $y$ are not all divisible by $3$. Thus we just need
to consider the following two cases.

{\it Case} b1. $3\nmid x$.

When this occurs, without loss of generality, we may assume that $x\equiv -z\pmod 3$ (otherwise we may replace $z$
be $-z$). In view of the identity
\begin{equation}\label{5.7}
g_1\l(x-3y,\ \frac{x-2z}{3},\ \frac{x+z}{3}\r)=g_2(x,y,z),
\end{equation}
there are $x^*,y^*,z^*\in\Z$ such that $15n+\delta=g_1(x^*,y^*,z^*)$.

{\it Case} b2. $3\mid x$ and $3\nmid y$

In this case, with the help of the identity
\begin{equation*}
g_2(x-y,\ -y,\ z)=g_2(x,y,z),
\end{equation*}
we return to Case b1 since $x-y\not\eq0\pmod3$.

Now applying Lemma {Lem2.1} we immediately obtain the $(15,r)$-universality of $x^2+5y^2+30z^2$.

(c) For any $r\in\{4,11,14\}$, it is easy to verify that $15n+r$ can be
represented by $h_1(x,y,z)=x^2+10y^2+15z^2$ locally. There are two classes in the genus of $h_1$,
and the one not containing $h_1$ has a representative $h_2(x,y,z)=5x^2+5y^2+6z^2$.

Suppose that $15n+r=h_2(x,y,z)$ with $x,y,z\in\Z$. Since $15n+r\ne 6z^2$, by Lemma \ref{Lem5.4} and the symmetry of $x$ and $y$, we simply assume $5\nmid y$ without loss of generality.
We claim that we may adjust the signs
of $x,y,z$ to satisfy the congruence $(2x+y+2z)(x-2y-2z)\equiv0\pmod5$.

{\it Case} c1. $x^2\equiv y^2\pmod 5$.

Without loss of generality, we may assume $x\equiv y\equiv z\pmod 5$ if $x^2\equiv y^2\equiv z^2\pmod 5$,
and $x\equiv-y\equiv-2z\pmod 5$ if $x^2\equiv y^2\equiv -z^2\pmod 5$. So our claim holds in this case.

{\it Case} c2. $x^2\equiv -y^2\pmod 5$.

If $x^2\equiv -y^2\equiv z^2\pmod 5$, without loss of generality,
we may assume that $x\equiv -2y\equiv z\pmod 5$. If $x^2\equiv -y^2\equiv -z^2\pmod 5$, we may assume that $x\equiv-2y\equiv2z\pmod 5$ without loss of any generality.
Thus $x,y,z$ satisfy the desired congruence in our claim.

{\it Case} c3. $x^2\equiv 0\pmod 5$.

If $y^2\equiv z^2\pmod 5$, we may assume that $y\equiv -z\pmod 5$.
If $y^2\equiv -z^2\pmod 5$, without loss of generality, we may assume that $y\equiv -2z\pmod 5$. Clearly, our claim also holds in this case.

In view of the above analysis, there are $x,y,z\in\Z$ with $2x+y+2z\equiv 0\pmod 5$ or $x-2y-2z\equiv 0\pmod 5$ such that
$15n+r=h_2(x,y,z)$. One may easily verify the following identities
\begin{align}
h_1\l(x-2y,\ \frac{2x+y+2z}{5}-z,\ \frac{2x+y+2z}{5}\r)=&h_2(x,y,z)\label{5.8},
\\h_1\l(x+2y,\ \frac{x-2y-2z}{5}+z,\ \frac{x-2y-2z}{5}\r)=&h_2(x,y,z)\label{5.9}.
\end{align}
With the help of (\ref{5.8}), (\ref{5.9}) and Lemma \ref{Lem2.1}, the $(15,r)$-universality of
$x^2+10y^2+15z^2$ is valid.\qed
\medskip

\begin{lemma}\label{Lem5.5}
Let $n\in\N$ and $g_2(x,y,z)=2x^2+8y^2+15z^2-2xy$.
Assume that $15n+8=g_2(x,y,z)$ for some $x,y,z\in\Z$ with
$y^2+z^2\ne0$. Then $15n+8=g_2(u,v,w)$
for some $u,v,w\in\Z$ with $3\nmid v+w$.
\end{lemma}
\begin{proof}
Suppose that every integral solution of the equation $g_2(x,y,z)=15n+8$ satisfies
$y+z\equiv 0\pmod 3$. We want to deduce a contradiction.

Let
$$T=\begin{pmatrix} 1 & -2/3 &-2/3 \\0 & -1/3 &-4/3 \\ 0 & 2/3 &-1/3 \end{pmatrix}$$
and let $V$ be the quadratic space corresponding to $g_2$.
Since
\begin{equation*}
g_2\l(x+\frac{-2y-2z}{3},\ \frac{-y-4z}{3},\ \frac{2y-z}{3}\r)=g_2(x,y,z).
\end{equation*}
We have $T\in O(V)$. One may
easily verify that the order of $T$ is infinite and the space $V_T$ defined  in Lemma \ref{Lem5.1}
coincides with $\{(t,0,0): t\in\Q\}$.
By the assumption in the lemma, we have $15n+8=g_2(x_0,y_0,z_0)$ for some $(x_0,y_0,z_0)\in\Z^3\sm V_T$.
Note that the set $\{T^k(x_0,y_0,z_0): k\gs0\}$ is infinite and all elements of this set are solutions to the equation $g_2(x,y,z)=15n+8$.
This leads to a contradiction since the number of integral representations of any integer by a positive
quadratic forms is finite.
\end{proof}

\noindent{\it Proof of Theorem} 1.4(iii). (a) Let $r\in\{8,11,14\}$. It is easy to see that $15n+r$ can be represented by
$f_1(x,y,z)=3x^2+5y^2+6z^2$ locally. There are two classes in the genus of $f_1$, and the one
not containing $f_1$ has a representative $f_2(x,y,z)=2x^2+6y^2+9z^2+6yz$.

Suppose that $15n+r=f_2(x,y,z)$ for some $x,y,z\in\Z$. Then $3\nmid x$.

{\it Case} 1. $3\nmid y$.

In this case, without loss of generality we may assume that $x\equiv -y\pmod 3$ (otherwise we may replace
$x$ by $-x$). In view of the identity
\begin{equation}\label{5.11}
f_1\l(\frac{2x-y}{3}-z,\ -y,\ \frac{x+y}{3}+z\r)=f_2(x,y,z),
\end{equation}
there are $x^*,y^*,z^*\in\Z$ such that $15n+\delta=f_1(x^*,y^*,z^*)$.

{\it Case} 2. $y^2+z^2\ne0$ and $3\mid y$.

When this occurs, by Lemma \ref{Lem5.3} we may simply assume that $3\nmid z$. With the help of the identity
\begin{equation*}
f_2(x,\ y+z,\ -z)=f_2(x,y,z),
\end{equation*}
we return to Case 1.

{\it Case} 3. $15n+r=2m^2$ for some $m\in\N$.

When $15n+r=2m^2=2\times 2^{2k}$ with $k\gs1$, we have
\begin{equation*}
15n+r=2\times 2^{2k}=3\times(2^{k-1})^2+5\times(2^{k-1})^2+6\times 0^2.
\end{equation*}

Now suppose that $m$ has a prime factor $p>5$. By Lemma \ref{Lem2.3} we have
\begin{equation}\label{5.12}
r(2p^2,f_1)+r(2p^2,f_2)=2\l(p+1-\left(\frac{-5}{p}\right)\r)>10.
\end{equation}
Clearly $r(2p^2,f_1)>5$ or $r(2p^2,f_2)>5$. When $r(2p^2,f_1)>5$,
the number $2m^2$ can be represented by $f_1$ over $\Z$ since $r(2m^2,f_1)\gs r(2p^2,f_1)$. When $r(2p^2,f_2)>5$, there are $u,v,w\in\Z$
with $v^2+w^2\ne0$ such that $f_2(u,v,w)=15n+r$. By Lemma \ref{Lem5.3}, we return to Case 1 or Case 2.

In view of the above, by applying Lemma \ref{Lem2.1} we get the $(15,r)$-universality of $3x^2+5y^2+6z^2$.

(b) It is easy to see that $15n+8$ can be represented by the genus of $g_1(x,y,z)=3x^2+5y^2+15z^2$.
There are two classes in the genus of $g_1$, and the one not containing $g_1$ has a representative
$g_2(x,y,z)=2x^2+8y^2+15z^2-2xy$.

Suppose that the equation $15n+8=g_2(x,y,z)$ is solvable over $\Z$.
We claim that there are $u,v,w\in\Z$ with $(u+w)(u-v-w)\equiv 0\pmod 3$ such that $15n+8=g_2(u,v,w)$.

{\it Case} 1. $3\mid x$.

Clearly, $3\nmid y$. If $3\nmid z$, without loss of generality we may assume that
$z\equiv -y\pmod 3$ (otherwise we replace $z$ by $-z$). Then $(u,v,w)=(x,y,z)$ meets our purpose.

{\it Case} 2. $3\nmid x$ and $y^2+z^2\ne0$.

In this case, by Lemma \ref{Lem5.5} there are $x',y',z'\in\Z$ with $3\nmid y'+z'$ such that
$15n+8=g_2(x',y',z')$. If $x'\equiv y'\pmod 3$, then by using the identity
\begin{equation*}
g_2(x-y,\ -y,\ z)=g_2(x,y,z),
\end{equation*}
we return to Case 1. If $x'\not\equiv y'\pmod 3$, then $3\mid y'$ and $3\nmid z'$ since $(x'+y')^2\equiv 1\pmod 3$
and $3\nmid x'$. Without loss of generality, we may assume that
$x'\equiv -z'\pmod 3$. So $(u,v,w)=(x',y',z')$ meets our purpose.

{\it Case} 3. $15n+8=2m^2$ with $m\in\N$.

If $m=2^k$ for some $k\in\Z^+$, then
\begin{equation*}
2m^2=3\times (2^{k-1})^2+5\times (2^{k-1})^2+15\times 0^2.
\end{equation*}

Now suppose that $m$ has a prime factor $p>5$. By Lemma \ref{Lem2.3}, we have
\begin{equation}\label{5.13}
r(2p^2,g_1)+r(2p^2,g_2)=2\l(p+1-\left(\frac{-2}{p}\right)\r)>10.
\end{equation}
Clearly, $r(2p^2,g_1)>5$ or $r(2p^2,g_2)>5$. When $r(2p^2,g_1)>5$,
we have $r(2m^2,g_1)\ge r(2p^2,g_1)>5$. If $r(2p^2,g_2)>5$, then there exist $x_0,y_0,z_0\in\Z$
with $y_0^2+z_0^2\ne 0$ such that $15n+8=g_2(x_0,y_0,z_0)$. So we are reduced to previous cases.

In view of the proved claim, the $(15,8)$-universality of $g_1$ follows from Lemma \ref{Lem2.1} and the identities
\begin{align*}
g_1\l(\frac{x-5z}{3}-y,\ -y+z,\ -\frac{x+z}{3}\r)&=g_2(x,y,z),\\
g_1\l(\frac{x-y-z}{3}+y+2z,\ y-z,\ \frac{x-y-z}{3}\r)&=g_2(x,y,z).
\end{align*}

(c) One may easily verify that $15n+8$ can be represented by $h_1(x,y,z)=3x^2+5y^2+30z^2$
locally. There are two classes in the genus of $h_1$, and the one not containing $h_1$ has a
representative $h_2(x,y,z)=2x^2+15y^2+15z^2$.

Suppose that the equation $15n+8=h_2(x,y,z)$ for some $x,y,z\in\Z$.
In light of Lemma \ref{Lem5.4}, we may assume $5\nmid y$ if $y^2+z^2>0$.
We claim that there are $u,v,w\in\Z$ with $(u-v+2w)(u-2v+w)\eq0\pmod5$ such that
$15n+8=h_2(u,v,w)$. This holds trivially if $y^2+z^2>0$ and $5\mid z$. Below we discuss the remaining cases.

{\it Case} 1. $y^2\equiv \ve z^2\pmod 5$ and $y^2+z^2\ne0$, where $\ve\in\{\pm1\}$.

 If $y^2\equiv z^2\equiv x^2\pmod 5$, then we may assume that $x\equiv y\equiv z\pmod 5$. If $y^2\equiv z^2\equiv -x^2\pmod 5$, without loss
of generality we may assume that $x\equiv -2y\equiv 2z\pmod 5$. So, $(u,v,w)=(x,y,z)$ meets our requirement in the case $\ve=1$. The case $\ve=-1$ can be handled similarly.

{\it Case} 2. $15n+8$ is twice a square, say $2m^2$ with $m\in\Z^+$.

When $m=2^k$ with $k\in\Z^+$, we have
\begin{equation*}
15n+8=2\times 2^{2k}=3\times (2^{k-1})^2+5\times (2^{k-1})^2+30\times 0^2.
\end{equation*}

Now assume that $m$ has a prime factor $p>5$. By Lemma \ref{Lem2.3}, we have
\begin{equation}
2r(2p^2,h_1)+r(2p^2,h_2)=2\l(p+1-\left(\frac{-1}{p}\right)\r)>10.
\end{equation}
Clearly $r(2p^2,h_1)\ge 4$ or $r(2p^2,h_2)\ge 4$. If $r(2p^2,h_1)\ge 4$,
then $r(2m^2,h_1)\ge r(2p^2,h_1)\ge 4$. When $r(2p^2,h_2)\ge 4$,
there exist $u,v,w\in\Z$ with $v^2+w^2\ne 0$ such that $h_2(u,v,w)=2m^2$.
Thus we are reduced to Case 1.

In view of the proved claim and the identities

\begin{align*}h_2(x,y,z)=&
h_1\l(2y+z,\ \frac{2x+3y-z}{5}-z,\ \frac{x-y+2z}{5}\r)
\\=&h_1\l(y+2z,\ \frac{2x+y-3z}{5}+y,\ \frac{x-2y+z}{5}\r),
\end{align*}
by applying Lemma \ref{Lem2.1} we obtain the $(15,8)$-universality of $3x^2+5y^2+30z^2$. \qed


\begin{thebibliography}{99}
\bibitem{AAW} A. Alaca, S. Alaca and K. S. Williams, {\it Arithmetic progressions and binary quadratic
    forms}, Amer. Math. Monthly {\bf 115} (2008), 252--254.
\bibitem{C} J. W. S. Cassels, {\it Rational Quadratic Forms}, Academic Press, London, 1978.
\bibitem{D39} L. E. Dickson, {\it Modern Elementary Theory of Numbers}, Univ. of Chicago Press,
    Chicago, 1939.
\bibitem{DW} G. Doyle and K. S. Williams, {\it  A positive-definite ternary quadratic form does not represent all positive integers}, Integers {\bf 17} (2017), \#A41, 19pp (eletronic).
\bibitem{GPS} S. Guo, H. Pan and Z.-W. Sun, {\it Mixed sums of squares and triangular numbers (II)}, Integers {\bf 7} (2007), \#A56, 5pp (electronic)
\bibitem{Jagy96} W. C. Jagy, {\it Five regular or nearly-regular ternary quadratic forms}, Acta Arith.
    {\bf 77} (1996), 361--367.
\bibitem{JKS} W. C. Jagy, I. Kaplansky and A. Schiemann, {\it There are 913 regular ternary forms},
    Mathematika {\bf 44} (1997), 332--341.
\bibitem{Jagy14} W. C. Jagy, {\it Integral Positive Ternary Quadratic Forms}, Lecture Notes, 2014.
\bibitem{Jones} B. W. Jones and G. Pall, {\it Regular and semi-regular positive ternary quadratic forms},
    Acta Math. {\bf 70}(1939), 165--191.
\bibitem{odd} I. Kaplansky, {\it Ternary positive quadratic forms that represent all odd positive
    integers}, Acta Arith. {\bf 70}(1995), 209--214.
\bibitem{Ki} Y. Kitaoka, {\it Arithmetic of Quadratic Forms}, Cambridge Tracts in Math., Vol. 106, 1993.
\bibitem{Oto} O. T. O'Meara, {\it Introduction to Quadratic Forms}, Springer, New York, 1963.
\bibitem{Oh} B.-K. Oh, {\it Ternary universal sums of generalized pentagonal numbers}, J. Korean Math.
    Soc. {\bf 48} (2011) 837--847.
\bibitem{PK} L. Pehlivan and K. S. Williams, {\it $(k,l)$-universality of ternary quadratic forms $ax^2+by^2+cz^2$},
Integers {\bf 18} (2018), \#A20, 44pp (eletronic).

\bibitem{S07} Z.-W. Sun, {\it Mixed sums of squares and triangular numbers}, Acta Arith. {\bf 127} (2007),
    103--113.
\bibitem{S15} Z.-W. Sun, {\it On universal sums of polygonal numbers}, Sci. China Math. {\bf 58} (2015),
    1367--1396.
\bibitem{S17} Z.-W. Sun, {\it On $x(ax+1)+y(by+1)+z(cz+1)$ and $x(ax+b)+y(ay+c)+z(az+d)$},
J. Number Theory {\bf 171} (2017), 275--283.
\bibitem{S17o} Z.-W. Sun, Sequence A286885 on OEIS (On-Line Encyclopedia of Integer Sequences),
{\tt http://oeis.org/A286885}, August 2, 2018.

\bibitem{List} Z.-W. Sun, {\it On universal sums $x(ax+b)/2+y(cy+d)/2+z(ez+f)/2$},
Nanjing Univ. J. Math. Biquarterly {\bf 35} (2018), 85--199.
\bibitem{Wu-Sun} H.-L. Wu and Z.-W. Sun, {\it Some universal quadratic sums over the integers},
Electron. Res. Arch. {\bf 27} (2019), 69--87.

\end{thebibliography}
\end{document}